\documentclass[12pt]{article}

\usepackage{dsfont}
\usepackage{amsfonts,amsmath,amsthm}
\usepackage{amssymb}
\usepackage{algorithm}
\usepackage{algpseudocode}
\usepackage{xcolor}
\usepackage{graphicx}
\usepackage{stmaryrd}        
\usepackage{multirow}
\usepackage{bm}       
\usepackage{makecell}

\usepackage[paper=a4paper,dvips,top=2cm,left=2cm,right=2cm,
foot=1cm,bottom=4cm]{geometry}

\newcommand{\red}[1]{\begin{color}{black}#1\end{color}}

\newcommand{\qu}[1]{\tilde{#1}}

\newcommand{\dq}[1]{\hat{#1}}
\newcommand{\vdq}[1]{\hat{\mathbf{#1}}}
\newcommand{\dqm}[1]{\hat{#1}}
\newcommand{\dqset}{\hat{\mathbb Q}}
\newcommand{\udqset}{\hat{\mathbb U}}

\newcommand{\ii}{\mathbf{i}}
\newcommand{\jj}{\mathbf{j}}
\newcommand{\kk}{\mathbf{k}}

\newcommand{\vx}{\mathbf{x}}

\newcommand{\err}{\mathrm{err}}

\newtheorem{Thm}{Theorem}[section]

\newtheorem{Lem}{Lemma}
\newtheorem{Prop}[Thm]{Proposition}

\begin{document}
		\title{A Dual Quaternion Control Law for Formation Control of Multiple 3-D  Rigid Bodies} 
 \author{
 	Chunfeng Cui\footnote{School of Mathematical Sciences, Beihang University, Beijing  100191, China.
 		({\tt chunfengcui@buaa.edu.cn})}
 	\and 
 	Liqun Qi\footnote{
 		Department of Applied Mathematics, The Hong Kong Polytechnic University, Hung Hom, Kowloon, Hong Kong.
 		({\tt maqilq@polyu.edu.hk})}
 	\and Hao Chen \footnote{College of Mechatronics and Automation, National University of Defence Technology, Changsha, 410073, China.
 		({\tt  chenhao09@nudt.edu.cn})}
 			\and Xiangke Wang \footnote{College of Mechatronics and Automation, National University of Defence Technology, Changsha, 410073, China.
 			({\tt  xkwang@nudt.edu.cn})}
 }

 \date{\today}
 \maketitle
 
 \begin{abstract}
	\red{This paper studies the integrated position and attitude control problem for multi-agent systems of 3D rigid bodies. 
	While the state-of-the-art method in [Olfati-Saber and Murray, 2004] established the theoretical foundation for rigid-body formation control, it requires all agents to asymptotically converge to identical positions and attitudes, limiting its applicability in scenarios where distinct desired relative configurations must be maintained.} In this paper, we develop a novel dual-quaternion-based framework that generalizes this paradigm. By introducing a unit dual quaternion directed graph (UDQDG) representation, we derive a new control law through the corresponding Laplacian matrix, enabling simultaneous position and attitude coordination while naturally accommodating directed interaction topologies. Leveraging the recent advances in UDQDG  spectral theory, we prove global asymptotic convergence to desired relative configurations modulo a right-multiplicative constant and establish an $R$-linear convergence rate  determined by the second smallest eigenvalue of the UDQDG Laplacian.   	A projected iteration method is proposed to compute the  iterative states.  Finally, the proposed solution is verified  by several numerical  experiments.

 \end{abstract}

		\section{Introduction}
		
		\red{The simultaneous coordination of position and attitude  in multi-agent systems composed of 3D rigid bodies is important \cite{OPA15} and has been  widely applied in various domains, including autonomous mobile robots \cite{RB08}, unmanned aerial vehicles (UAVs) \cite{SKN11}, autonomous underwater vehicles (AUVs) \cite{YXL21}, and small satellite constellations \cite{CHWG24}. Traditional formation control methods   mostly treat each agent as a point-mass  and then employ consensus-based approaches to position them at desired locations within the formation. However, in multi-agent  cooperative missions such as search and rescue operations, to fully utilize payload capabilities, it is essential to simultaneously consider both the position and attitude of each agent (e.g., when multiple drones perform coordinated firefighting from different locations with varying orientations).   The inherent non-Euclidean geometry of the configuration space introduces unique complexities compared to point-mass agent models, necessitating geometric control theories and  novel formulations to guarantee convergence and stability.}

		For a rigid body in 3D space, the pose  is  characterized by its rotational and translational components. The translational component is naturally described by a vector $\bm{t}\in \mathbb R^3$. For rotation representation, multiple mathematical formulations exist, each with distinct advantages and limitations: Euler angles and axis-angle provide a minimal parametrization but suffer from  singularities;   rotation matrices in the special orthogonal group $SO(3)$ maintain orthogonality constraints; while unit quaternions $\mathbb{U}$ provide a computationally efficient and singularity-free representation.   	The complete rigid body transformation can be comprehensively represented using several  mathematical frameworks: the  Lie algebra $se(3)$  is storage efficient yet singular; the special Euclidean group $SE(3)$ combines $SO(3)$ and translation;  and unit dual quaternions $\udqset$ offer a compact,  singularity-free representation with advantageous algebraic properties.  
			Among these representations, dual quaternions exhibit particularly advantageous properties. Their scalar-like algebraic structure facilitates an elegant and mathematically  formulation of the dual quaternion Laplacian matrix, where each element is inherently a dual quaternion number. This stands in contrast to alternative representations, which are constrained to either vector or matrix forms.
			
			\begin{table}[t]
				\centering
				\footnotesize
				\renewcommand{\arraystretch}{1.3}
				\caption{Rotation and configuration representations in 3D multi-agent formation control.The abbreviations `Y' and `N'   represent `Yes' and `No', and `Cons' represent `Constraints', respectively.}
				\resizebox{\linewidth}{!}{
					\begin{tabular}{|c|c|c|c|c|c|c|}
						\Xhline{1.2pt}
						& & Definition  & Storage & Cons & Singular & Scalar\\
						\hline
						\multirow{4}{*}{Rotation}	&	Euler angles  &   $(\phi,\theta,\varphi)$            & 3      & N       &    Y    &    N      \\
						&	Axis-angle  &   $so(3):=\{\bm{v} \in \mathbb{R}^{3}\}$   & 3              &    N    & Y    &    N   \\
						&	Special orthogonal group  &   $SO(3):=\{R \in \mathbb{R}^{3 \times 3} \mid R^\top R = I_{3}, \operatorname{det}(R)=1\}$      & 9              &    Y    & N    &    N    \\
						&	Unit quaternion  &   $\mathbb{U} :=\{{q}\in \mathbb{Q}\mid|{q}|=1\}$             & 4         &    Y    & N  &   Y   \\ 
						\Xhline{1pt} 
						&		Lie algebra   &        $se(3):=\{(\bm{\rho},\bm{t})\mid \bm{\rho}\in so(3),\bm{t}\in\mathbb{R}^{3}\}$    &  6&          N          &  Y     &    N       \\
						Configuration	&	Special Euclidean group   &        $SE(3):=\{(R,\bm{t})\mid R\in SO(3),\bm{t}\in\mathbb{R}^{3}\}$    &12&          Y          &  N        &    N    \\ 
						&	Unit dual quaternion   &        $\udqset:=\{{q}_r+{q}_d\epsilon\in\hat{\mathbb{Q}}\mid{q}_r\in\mathbb{U}, \left\langle{q}_r,{q}_d\right\rangle=0\}$    &  8&          Y          &  N        &   Y    \\
						\Xhline{1.2pt}
				\end{tabular}}
				\label{table-pose-representation}
			\end{table}

			In the 3D formation control problem, the communication topology of networked agents plays a pivotal role in system coordination. This topological structure can be mathematically characterized by the graph Laplacian matrix – a fundamental tool for control law design that has proven particularly effective in solving consensus problems in multi-agent systems.  In their seminal work in 2004 and 2007, Olfati-Saber and Murray \cite{OM04}, complemented by   Olfati-Saber, Fax, and Murray \cite{OFM07}, pioneered a control framework specifically designed to solve consensus problems in multi-agent systems with switching network topologies and time-delay characteristics. Their approach can be formulated as follows:
		\begin{equation}\label{con:OM04}
			\dot \vx(t) = -L \vx(t),
		\end{equation}
		where $\vx(t)$ represents the state vector of the  nodes    and $L$ denotes the graph Laplacian matrix. Olfati-Saber and Murray \cite{OM04} showed that if \red{the underlying graph} $G$ is either a connected undirected graph or a strongly connected digraph, then the algorithm in \eqref{con:OM04}  guarantees asymptotic consensus for all initial states. In this context, consensus is achieved if and only if  $x_i=x_j$ for all $i\neq j$, indicating state synchronization across all network nodes. 
		Moreover, the study established a fundamental relationship between the network's algebraic connectivity and the control performance. 
		In 2005, Ren and Beard \cite{RB05}  investigated the consensus seeking in multi-agent systems under dynamically changing interaction topologies.

		For planar state spaces, the dynamics of a system   are constrained to a 2D plane, simplifying modeling and control. The \red{coordinate-free} consensus and formation problems have been effectively addressed using complex Laplacians and rotation matrices, as demonstrated in \cite{LWHF14} and \cite{LH15}. These approaches have proven particularly valuable in handling the geometric constraints inherent in planar configurations  
		\red{and highlighted the important role of the graph Laplacian in control law design.}

		When addressing the pose representation of 3-D rigid bodies, dual quaternions have gained widespread adoption owing to their significant advantages over alternative representations.  
		The application of dual quaternions in pose representation and control has been extensively explored in recent research. 
		In 2012, Wang et al.  \cite{WHYZ12,WYL12} pioneered the use of dual quaternions as a mathematical framework for pose representation and developed a novel control law capable of achieving network consensus.  In 2016, Gui and Vukovich \cite{GV16} introduced an innovative dual-quaternion-based adaptive control scheme for spacecraft motion tracking, demonstrating reduced control effort requirements. Recent advancements in this field have further expanded the application of dual quaternion-based approaches \cite{FA16,KFIA17,MK17,MMGB17,QWC25,WY17}.
		
		Especially, in 2020, Savino et al. \cite{SPSA20} made significant contributions by presenting an effective solution to the pose consensus problem for multiple mobile manipulators, i.e.,
		\begin{equation}\label{control:SPSA20}
			\dot{\vdq y} = -L \vdq y,
		\end{equation}
		where $\vdq y$ represents a dual quaternion vector, \red{$L$ is  the Laplacian matrix of the underlying graph.} 
		\red{The proposed pose consensus protocol guarantees convergence under the condition that the inter-agent interactions are characterized by directed graphs containing directed spanning trees \cite{SPSA20}. 
			However, the    matrix  $L$ involved in \eqref{control:SPSA20} is still the   Laplacian of the underlying graph. Consequently, all agents converge to identical configurations, and \eqref{control:SPSA20} is not applicable in scenarios where agents must reach distinct target states, such as 
			distributed tracking   and  consensus with heterogeneous states \cite{GHMHC22}.

			In all the approaches mentioned above, the elements of the Laplacian matrix are either real- or complex-valued quantities. However, for 3D rigid bodies, dual quaternions offer a distinct advantage—unlike vector or matrix representations, a dual quaternion is a scalar-like quantity, facilitating the formulation of a novel dual quaternion-based Laplacian matrix.
			
			Although the use of dual quaternion Laplacian matrices in formation control problems has not yet been widely investigated, recent advances in the spectral theory of dual quaternion matrices \cite{CLQW24,QCO24,QL23} have now established the theoretical foundation.   
			In 2023, Qi, Wang, and Luo \cite{QWL23}   proposed the dual quaternion Laplacian matrix  and   investigated  their pivotal role in formation control systems.    Subsequently, Chen et al. \cite{CHWG24} conducted an in-depth investigation into   precise formation flying problem in satellite clusters.}

		In this paper, we explore the formulation of control laws utilizing dual quaternion Laplacian matrices and establish the convergence theory through  the eigenvalues associated with the UDQDG Laplacian matrix. 
		Our main contributions are summarized as follows:
		\begin{itemize}
			\item[(i)] We present a novel dual-quaternion-based control law for 3D rigid-body formations using UDQDG Laplacian matrices, and prove the global asymptotic convergence and $R$-linear convergence rate.

			\item[(ii)] We introduce a projected iterative method to numerically compute agent states within the unit dual quaternion manifold at each iteration.  
			
			\item[(iii)] Through systematic numerical experiments, we confirm the  $R$-linear convergence rates that corroborate our theoretical findings.
		\end{itemize}
		
		The remainder of this paper is organized as follows.
		In next section, we provide the necessary preliminaries on dual quaternions and formation control.
		Section~\ref{sec:mainresults} presents our main results, including the proposed control law and its theoretical analysis. 
		\red{In Section~\ref{sec:Proj_method}, we develop a projected iterative method to compute agent states while preserving unit dual quaternion constraints.}   
		Numerical experiments are detailed in  Section~\ref{sec:numer}, demonstrating the efficiency of our approach. Finally, we conclude this paper in Section~\ref{sec:conclu}. 
		
		\section{Preliminary}

		\subsection{Directed graph}

		Consider the formation control problem for a system of $n$ rigid bodies, which can be modeled using a directed graph    $G=(V,E)$, where  $V$ represents the set of  $n$ nodes and
		$E$ denotes the set of $m$ directed arcs.
		For any two rigid bodies $i$ and $j$ in $V$, if rigid body $i$ can sense rigid body $j$, then  {arc} $(i, j) \in E$,  where $i$ and $j$ are the tail and head of the arc $(i, j)$, respectively. 
		Note that the existence of  $(i, j) \in E$  does not imply $(j , i) \not \in E$, as the graph is directed. 
		The Laplacian matrix of  $G$ is then defined as follows:
		\begin{equation}\label{Laplaican_mat_G}
			L =   {D- A},
		\end{equation}
		where  {$D$} is a diagonal real matrix with    $d_i$ representing  the number of  {arcs  {going} out from} $i$
		and ${A}=(a_{ij})$ is the adjacency matrix with  $a_{ij}=1$ if $(i,j)\in E$ and $a_{ij}=0$ otherwise.

		A directed spanning tree of a digraph $G$ is a subgraph, wherein one node, designated as the root, possesses no in-neighbors, and from this root, a path exists to every other node.
		The following result, taken from \cite{RB08}, proves indispensable in our analysis.

		\begin{Lem}[\cite{RB08}, Corollary 2.5] \label{Lem:zeroeig_G}
			The nonsymmetrical Laplacian matrix  of a directed graph has a simple zero eigenvalue with an associated eigenvector ${\bf 1}$  and all of the other eigenvalues are in the open right half plane if and only if the digraph has a directed spanning tree.
		\end{Lem}


		\subsection{Dual quaternion}
		
		The concept of quaternions was pioneered by W. R. Hamilton in 1843. Let $\mathbb Q$  denote the ring of quaternions.  A quaternion, denoted as  $q=q_0+q_1\ii + q_2\jj +q_3\kk$, comprises three imaginary components, where the imaginary units satisfy the fundamental relations   $\ii^2=\jj^2=\kk^2=\ii\jj\kk=-1$. 
		The   conjugate of a quaternion $q$  is denoted by $q^*=q_0-q_1\ii - q_2\jj -q_3\kk$ 
		and the magnitude is given by $|q|=\sqrt{q_0^2+q_1^2+q_2^2+q_3^2}$. 
		For notational convenience, a quaternion can also be represented as a four-dimensional vector  $[q_0,q_1,q_2,q_3]$. 
		
		Dual numbers were first introduced by W. K. Clifford in 1873 \cite{Cl73}. Let $\hat{\mathbb Q}$  denote the ring of dual quaternions.  A dual quaternion $\hat q \in\dqset$ comprises two quaternions 
		$$\hat q = q_s + q_d\epsilon\in\dqset,$$
		where   $\epsilon$ is the infinitesimal unit, satisfying $\epsilon \not = 0$ but $\epsilon^2 = 0$,  $q_s$ and $q_d$ are referred to as the standard  and  dual part of $\hat q$, respectively. 
		\red{The   conjugate of a dual quaternion $\hat q$ is defined by $\hat q^*=q_s^*+q_d^*\epsilon$.} 
		A dual quaternion $\hat q = q_s + q_d\epsilon$
		is called a {unit dual quaternion} (UDQ) if $|\hat q| = 1$, i.e.,
		$|q_s|=1 \text{ and } q_s^*q_d + q_d^*q_s =0.$
		The set of unit dual quaternions is denoted by $\hat {\mathbb U}$.

		A UDQ  can represent the {movement} of a rigid body as
		\begin{equation} \label{e3}
			\hat q = q_s + {1 \over 2} p^wq_s\epsilon ,
		\end{equation}
		where {$q_s$} is a unit quaternion, representing the rotation of the rigid body, and $p^w$  
		represents the {translation} of the rigid body. 
		Here, the superscript $w$ relating to the world frame \cite{WYL12}.  
		We may also use {$\hat q$} to represent the {configuration} of the rigid body.  Then {$q_s$} represents the {attitude} of the rigid body, and $p^w$  {represents} the {position} of the rigid body.   For further details on quaternions and dual quaternions, please refer to \cite{QWL23,WYL12}.
		
		\red{For any $n$-dimensional dual quaternion vector $\vdq x=\vx_s+\vx_d\epsilon=(\dq x_i)\in \dqset^n$,  let  $\vx_s^*$ denote   its conjugate transpose, and  $\bar \vx$  denote the component-wise conjugation, respectively.  
			The 2-norm of $\vdq x$ is defined    as $\|\vdq x\|=	\|\vx_d\|\epsilon$ if $\vx_s= {\bf 0}_n$ and $\|\vdq x\|=\|\vx_s\|+\frac{\vx_s^*\vx_d+\vx_d^*\vx_s}{\|\vx_s\|}\epsilon$ otherwise \cite{QLY22}.  
			Furthermore, we   adopt the $2^R$ norm defined by 
			\begin{equation}\label{equ:2R-norm}
				\|\vdq x\|_{2^R} = \sqrt{\|\vdq x_s\|^2+\|\vdq x_d\|^2}
			\end{equation}   in the numerical experiments.} 
		
		\subsection{Unit dual quaternions and formation control}
		In the formation control  of $n$ rigid bodies, the relative configuration of rigid body are  represented by UDQs.
		Let $\hat q_{d_{ij}}$ denote the weight associated with  the arc $(i, j)$, which defines  a  unit dual quaternion directed graph (UDQDG)  $\Phi = (G, \udqset, \varphi)$ \cite{QCO24}, where $G = (V, E)$ is the underlying directed graph, and $\varphi(i, j) = \hat q_{d_{ij}}$ for any $(i, j) \in E$. 
		If $\hat q_{d_{ij}}=\hat q_{d_{ij}}^*$ for all $(i, j) \in E$, the UDQDG reduced to a dual quaternion unit gain graph (DQUGG). 
		The desired relative configuration scheme $\left\{ \hat q_{d_{ij}}\right\}$ is {reasonable} if and only if there is a desired formation $\vdq q_d\in {\hat{\mathbb U}^n}$ that satisfies
		\begin{equation} \label{desrel}
			\hat q_{d_{ij}} = \hat q_{d_i}^*\hat q_{d_j},\ \forall (i, j) \in E, 
		\end{equation}
		as established in \cite{QCO24}.
		This condition is intrinsically linked to the balance   of DQUGGs  \cite{CLQW24}. 	Specifically,  by exploring the eigenvalue problem of dual Hermitian matrices, and its link with DQUGG, Qi and Cui \cite{QC24} proposed a cross-disciplinary approach to verify the  reasonable of the relative configuration problem.

		The Laplacian matrix  of a UDQDG   $\Phi$ is  defined by
		\begin{equation}\label{Laplaican_mat}
			\dqm L =   {D-\dqm A},
		\end{equation}
		where  {$D$} represents  the diagonal   matrix 
		and ${\dqm A}=(\dq a_{ij})$ is the adjacency matrix defined by 
		\begin{equation*}
			\dq a_{ij} = \left\{
			\begin{array}{cl}
				{\hat q_{d_{ij}}},  &  \text{if } (i,j)\in E,\\
				0, & \text{otherwise}.
			\end{array}
			\right.
		\end{equation*}
		Recently,    the spectral properties of  DQUGGs \cite{CLQW24,QC24} and UDQDGs \cite{QCO24} are extensively studied.  
		Especially, Qi, Cui, and Ouyang \cite{QCO24} showed	if a desired relative configuration scheme is reasonable, then there is a diagonal dual quaternion  matrix $\dqm Q$ such that
		\begin{equation}\label{L=QLQ}
			\dqm L = \dqm Q^*L \dqm Q,
		\end{equation}
		where $L$ is the Laplacian matrix of the underlying graph $G$.

		\section{Convergence Analysis}\label{sec:mainresults}
		\subsection{A control law}
		\red{ 
			Consider a multi-agent system of $n$ 3D rigid bodies, where the inter-agent kinematics and interactions are modeled by a   unit dual quaternion directed graph (UDQDG)  $\Phi = (G, \udqset, \varphi)$. 
			To achieve the simultaneous coordination of position and attitude,}  we propose the following control law
		\begin{equation} \label{control_law}
			\dot{\vdq z}=-\dqm K\dqm L\vdq z,
		\end{equation}
		where $\dqm L$ denote  the corresponding Laplacian matrix defined by  \eqref{Laplaican_mat}. 
		Or equivalently, 
		\begin{equation}\label{equ:control_law}
			\dot{\hat z}_i= \dq k_i\left(\sum_{j\in \mathcal N_i} {\dq q_{d_{ij}}} \dq z_j-d_i\dq z_i \right)
		\end{equation}
		for any $i=1,\dots,n$. \red{Here,  $\dqm K$ is a precondition matrix, and $\mathcal N_i=\{j|(i,j)\in E\}$.

			The proposed control law \eqref{control_law}, while structurally analogous to classical consensus approaches \cite{OM04,OFM07, SPSA20}, undergoes a fundamental transformation since dual quaternions exhibit non-commutative multiplication, a property inherited from their quaternion components. 
			In fact, \eqref{control_law} generalizes several state-of-the-art methods. Specifically:
			\begin{itemize}
				\item[(i)] If  all desired relative configurations are set to the identity dual quaternion, i.e., $\dq q_{d_{ij}}\equiv\dq 1$, 
				then the dual quaternion Laplacian matrix $\dqm L$ reduces to the standard Laplacian matrix of the underlying graph $G$. Moreover, when  
				the initial configuration is  real-valued and  $\dqm K=I$,   \eqref{control_law} simplifies to \eqref{con:OM04} proposed by \cite{OM04}.
				
				\item[(ii)]  If $\dqm K=I$ and all desired relative configuration are  identical,	i.e.,  $\dq q_{d_1}=\dots=\dq q_{d_n}$ and $\dq q_{d_{ij}}=1$ {for all $(i,j)\in E$}, then the dual quaternion Laplacian matrix $\dqm L$ reduces to the standard Laplacian matrix of the underlying graph $G$. In this case, 
				\eqref{control_law} reduces to   \eqref{control:SPSA20} proposed by  \cite{SPSA20}.
			\end{itemize}
			
		}

		The following propositions are important in our proof. 
		\begin{Prop}\label{Prop:A1}
			Let $\vdq x=\vx_s+\vx_d\epsilon \in \dqset^n$, $\dq Q=Q_s+Q_d\epsilon\in \dqset^{n\times n}$, and $V\in\mathbb R^{n\times n}$. Then the following results hold. 
			\begin{itemize}
				\item[(i)] If $\dq Q^*\dq Q = I_n$, i.e., $Q_s^*Q_s=I_n$ and $Q_s^*Q_d+Q_d^*Q_s=O_n$, then  $\|\dq Q\vdq x\| = \|\vdq x\|.$ 
				
				\item[(ii)] $\|V\vdq x\| \le (1+\delta)\sigma_{\max}(V) \|\vdq x\|$ for arbitrary small positive constant $\delta$.  
			\end{itemize} 
		\end{Prop}
		\begin{proof}
			(i)  If $\vx_s\neq {\bf 0}_n$, then the product  $Q_s\vx_s$ is guaranteed to be non-zero  and 
			\begin{eqnarray*}
				&&\|\dq Q\vdq x\| \\ 
				&=&  \|Q_s\vx_s\|+\frac{(Q_s\vx_s)^*(Q_s\vx_d+Q_d\vx_s)}{\|Q_s\vx_s\|}\epsilon +\frac{(Q_s\vx_d+Q_d\vx_s)^*(Q_s\vx_s)}{\|Q_s\vx_s\|}\epsilon\\
				&=& \|\vx_s\| + \frac{\vx_s^*\vx_d+\vx_d^*\vx_s}{\|\vx_s\|}\epsilon  +\frac{\vx_s^*(Q_s^*Q_d+Q_d^*Q_s)\vx_s}{\|\vx_s\|}\epsilon\\
				&=& \|\vx_s\| + \frac{\vx_s^*\vx_d+\vx_d^*\vx_s}{\|\vx_s\|}\epsilon \\
				&=& \|\vdq x\|.
			\end{eqnarray*}
			Otherwise, if the standard part vanishes, i.e.,  $\vx_s= {\bf 0}_n$, then we may derive that   $\|\dq Q\vdq x\| = \|Q_s\vx_d+Q_d\vx_s\|\epsilon=\|Q_s\vx_d\|\epsilon=\|\vx_d\|\epsilon=\|\vdq x\|$. 
			
			(ii) Let $V\in\mathbb R^{n\times n}$. When $\vx_s= {\bf 0}_n$,   we may obtain the inequality  $\|V\vdq x\|  = \|V\vx_d\|\epsilon \le \sigma_{\max}(V)\|\vx_d\|\epsilon,$ where $\sigma_{\max}(V)$ is the maximum singular value of $V$. For the case  $\vx_s\neq {\bf 0}_n$,   we may  derive 
			\begin{eqnarray*}
				\|V\vdq x\|
				&=&  \|V\vx_s\|+\frac{\vx_s^*V^*V\vx_d+\vx_d^*V^*V\vx_s}{\|V\vx_s\|}\epsilon\\
				&\le & \sigma_{\max}(V)\|\vx_s\|+\frac{\vx_s^*V^*V\vx_d+\vx_d^*V^*V\vx_s}{\|V\vx_s\|}\epsilon\\
				&<& (1+\delta)\sigma_{\max}(V) \|\vdq x\|,
			\end{eqnarray*}
			where the last inequality is derived from the ordering principle of dual numbers \cite{QLY22}. Specifically, for any dual numbers    $a=a_s+a_d\epsilon$ and $b=b_s+b_d\epsilon$, the relation  $a<b$ holds if either  $a_s<b_s$ or $a_s=b_s$ and  $a_d<b_d$.

			This completes the proof. 
		\end{proof}

		\red{
			Next, we establish an upper bound for the matrix norm associated with the Jordan block. Notably, this result does not require the Laplacian matrix itself to be a Jordan block, but rather emerges from our intermediate  analysis.
		}
		\begin{Prop}\label{Prop:A2}
			Let $\lambda=\lambda_r+\ii \lambda_i$ be a given complex number, and  $J_{\lambda,n}$ be an $n$-dimensional Jordan matrix, i.e.,  $J_{\lambda,n} = \lambda I_n + I_{n}^+$ and   $I_{n}^+$ denotes a nilpotent matrix with 1's on the superdiagonal and 0's elsewhere.   
			Then for any $t\ge 0$, 	we have  $$\sigma_{\max}(\exp(-J_{\lambda,n} t))\le \red{\left(n-1+\frac{n}{(n-1)!}t^{n-1}\right)} \exp( -\lambda_r t).$$  
		\end{Prop}
		\begin{proof} Through direct computation, we obtain the matrix exponential factorization $\exp(-J_{\lambda,n} t) =\exp(-\lambda  t) \exp(-I_n^+ t)$. For an arbitrary real-valued vector $\vx\in \mathbb R^n$, we may derive that 
			\begin{eqnarray*} 
				\left\| \exp(-I_n^+ t) \vx \right\| = \left\|\sum_{k=0}^{n-1}\frac{(-I_n^+ t)^k}{k!}\vx\right\| \le \sum_{k=0}^{n-1}\frac{t^k}{k!}\|\vx\|.
			\end{eqnarray*}
			Consequently,  we have $\sigma_{\max} (\exp(-I_n^+ t))\le \sum_{k=0}^{n-1}\frac{t^k}{k!}.$ 
			We now analyze the inequality by considering two distinct temporal regimes  of $t$.

			\red{If $t\in [0,n-1)$,   we can derive that  $\frac{t^k}{k!}$ is monotonically decreasing from $k=0$ to $\lfloor t\rfloor$ and then monotonically increasing from $k=\lceil t\rceil$  to $n-1$.  Therefore, the following inequality holds  
				\begin{eqnarray*}
					\sum_{k=0}^{n-1}\frac{t^k}{k!}&\le&  1+\lfloor t\rfloor + (n-\lceil t\rceil)\frac{t^{n-1}}{(n-1)!} \le n-1+ \frac{n}{(n-1)!}t^{n-1}. 
				\end{eqnarray*}  
				Similarly, if $t\ge n-1$, we may derive that   $\sum_{k=0}^{n-1}\frac{t^k}{k!}\le  \frac{n}{(n-1)!}t^{n-1}$.

				Consequently, it holds that}  
			\begin{eqnarray*}
				\sigma_{\max}(\exp(-J_{\lambda,n} t))  
				&=&\exp( -\lambda_r t) \sigma_{\max} (\exp(-I_n^+ t))  \\ 
				&\le&  \left(n-1+\frac{n}{(n-1)!}t^{n-1}\right) \exp( -\lambda_r t). 
			\end{eqnarray*}
			This completes the proof. 
		\end{proof}

		In the following, we show the proposed control law \eqref{control_law} guarantees asymptotic convergence of the multi-agent system to the desired formation. The proof employs the Jordan canonical form to analyze the system dynamics.

		\begin{Thm}\label{Thm:controllaw}
			Suppose the   directed  graph $G = (V, E)$ contains a directed spanning tree and   there exists  $\vdq q_d\in{\hat{\mathbb U}^n}$  such that  (\ref{desrel}) holds for all $\left\{ \hat q_{d_{ij}} : (i, j) \in E \right\}$. 
			Let   $\dqm K= \dqm Q^*K \dqm Q$, where  $\dqm Q=\mathrm{diag}(\vdq q_d)$ and $K\in\mathbb R^{n\times n}$ is a positive diagonal matrix. Then  the following results hold. 
			\begin{itemize}
				\item[(i)]  The observed configuration vector $\vdq z = \vdq z(t)$   under control law \eqref{control_law}
				globally  and asymptotically  converges to
				$\vdq z^{\infty}=\bar{\vdq q}_d\dq c$, where  $\dq c\in {\hat {\mathbb Q}}$.
				\item[(ii)] Furthermore,  if   the initial configuration satisfies $\vdq z(0)\in\udqset^n$, then  $\|\vdq z(t)-\vdq z^{\infty}\|$ converges R-linearly to  zero with the rate  $\exp(-\lambda_{2r})$,  where  $\lambda_{2r}$ denotes the real part of the second smallest eigenvalue of $KL$.
			\end{itemize}
		\end{Thm}
		\begin{proof}  
			Through direct computation and and \eqref{L=QLQ}, we obtain  
			\[\dqm K \dqm L =( \dqm Q^* K \dqm Q)( \dqm Q^* L \dqm Q)= \dqm Q^* KL \dqm Q.\]
			According to  Lemma~\ref{Lem:zeroeig_G},    all eigenvalues of $L$  have positive real parts  except for a single zero eigenvalue.	
			Since $K\in\mathbb R^{n\times n}$ is a positive diagonal matrix, the scaled matrix  $KL$ maintains this eigenvalue structure.  
			Denote  the Jordan decomposition by $KL=SJS^{-1}$. Without loss of generalization, we express the Jordan matrix  as  $J=\mathrm{diag}(0,J_2)$, where $J_2$  represents a Jordan matrix whose diagonal entries all have positive real parts. 
			Denote ${\bf s}_1$ a the first column of $S$.
			From the fundamental property that    $\mathbf 1_n$ spans the null space of  $L$, we conclude that ${\bf s}_1$  must be collinear with   $\mathbf 1_n$.
			Let $V=\alpha S$, where the scaling factor   $\alpha=\frac{1}{{\bf 1}^\top {\bf s}_1}$ such that the first column of $V$ is exactly  ${\bf s}_1$.    Consequently, we obtain 
			\begin{equation}\label{KL_transform}
				V^{-1}KLV=\begin{bmatrix}
					0 & 0\\
					0 & J_2
				\end{bmatrix}.
			\end{equation} 
			Define
			\[{\vdq y}=\begin{bmatrix}
				\dq y_1\\
				{\vdq y_2}
			\end{bmatrix}=V^{-1}\dqm Q \vdq z.\]
			Then, the dynamics in \eqref{control_law} can be expressed as
			\begin{equation}
				\begin{bmatrix}
					\dot {\dq y}_1\\
					\dot{{\vdq y}}_2
				\end{bmatrix}= -\begin{bmatrix}
					0 & 0\\
					0 & J_2
				\end{bmatrix} \begin{bmatrix}
					\dq y_1\\
					{\vdq y_2}
				\end{bmatrix}.
			\end{equation}
			From this, we have $\dq y_1(t)=\dq y_1(0)$ remains constant, while ${\vdq y_2(t)=\exp(-J_2 t) {\vdq y_2(0)}}$ asymptotically  converges to a  zero vector as $t\rightarrow \infty$. Due to $V\vdq y (t)\rightarrow {\bf 1}_n \dq y_1(0)$ and $\dqm Q^*{\bf 1}_n=\bar{\vdq q}_d$, we derive that $\vdq z(t)={\dqm Q^*V\vdq y (t)}$ globally asymptotically converges to a formation $\bar{\vdq q}_d\dq y_1(0)$.

			Denote $\vdq z^{\infty}$ and $\vdq y^{\infty}$ as the limit state of $\vdq z(t)$ and $\vdq y(t)$, respectively. Under the assumption that the initial state satisfies   $\vdq z(0)\in\udqset^n$, we have 
			\begin{eqnarray*}
				\|\vdq z(t)-\vdq z^{\infty}\|  
				&=& \|\dqm QV\left( \vdq y(t)-\vdq y^{\infty}\right)\| \\
				&=& \|V\left( \vdq y(t)-\vdq y^{\infty}\right)\| \\
				&\le& (1+\delta)\sigma_{\max}(V)\| \vdq y(t)-\vdq y^{\infty}\| \\
				&=& (1+\delta)\sigma_{\max}(V)\| \vdq y_2(t)\| \\
				&\le& (1+\delta)^2\sigma_{\max}(V) \sigma_{\max}(\exp(-J_2 t))\| \vdq y_2(0)\| \\
				&\le&(1+\delta)^2 \sigma_{\max}(V) \red{c_{n_0}}\exp( -\lambda_{2r}t)\| \vdq y(0)\| \\ 
				&\le& (1+\delta)^3\kappa(V)\red{c_{n_0}}\exp( -\lambda_{2r}t)\| \vdq z(0)\|\\
				&=&  \sqrt{n}(1+\delta)^3\kappa(V)\red{c_{n_0}}\exp( -\lambda_{2r}t). 
			\end{eqnarray*}
			\red{Here, $c_{n_0} = \left(n_0-1+\frac{n_0}{(n_0-1)!}t^{n_0-1}\right)$, $n_0$ is the largest size of the Jordan block of $J_2$,} $\delta>0$ represents a small positive constant,   	$\kappa(V)$ corresponds to the condition number of  $V$, and $\lambda_{2r}$ characterizes the real part of the eigenvalue of  $KL$ with the second smallest real part.  The first, second, and fourth inequalities are established by Proposition~\ref{Prop:A1}, and the third inequality is a consequence of Proposition~\ref{Prop:A2} along with the fact that the spectral norm of a block diagonal matrix does not exceed the largest spectral norm among its individual blocks.   
			
			This implies that $\|\vdq z(t)-\vdq z^{\infty}\|$ exhibits R-linear convergence to zero at an asymptotic convergence rate of   $\exp( -\lambda_{2r})$, which completes the proof. 
		\end{proof}
		
		\red{Theorem~\ref{Thm:controllaw}  readily applies when the underlying graph $G$ is a connected undirected graph.
			In this case, the dual quaternion  Laplacian matrix $\dqm L$ becomes Hermitian and diagonalizable \cite{QL23}, with  algebraic multiplicity $n_0=1$ and $c_{n_0}=1$.}

		\section{A Projected Iteration Approach}\label{sec:Proj_method}
		\red{The control law in \eqref{control_law} inherently requires $\vdq z(t) \in \udqset^n$ for all $t \geq 0$ to maintain valid rigid-body configurations.   Hence, the algorithmic implementation requires innovative techniques, which represent significant departures from conventional vector-space consensus methods.  Specifically,} given an initial  configuration  $\vdq z(0)\in\udqset^n$, directly implementing the discrete version of \eqref{control_law} may result in an infeasible configuration. To address this issue, we  adopt a projected iteration approach in this paper, which ensures that the iterates remain within the feasible domain.
		
		Specifically,    for all iteration index $k=0,1,\dots$,  we update  $\vdq z(k)$ as follows,
		\begin{equation}\label{equ:proj_iterate}
			\vdq z(k+1) =  \mathrm{proj}_{\hat{\mathbb U}^n}\left(\vdq z(k) - \alpha_k\dqm K\dqm L\vdq z(k)\right),
		\end{equation}
		where $\alpha_k$ is the predefined stepsize, and   $\mathrm{proj}_{\hat{\mathbb U}^n}(\cdot)$ denotes the elementwise projection to the set of unit dual quaternions  \cite{CQ24}, i.e.,  
		for $\dq x=x_s+x_d\epsilon$, it holds that 
		\begin{align}\label{def:proj}
			\mathrm{proj}_{\hat{\mathbb U}}(\dq x) =\left\{
			\begin{array}{ll}
				\frac{x_s}{|x_s|} + \left(\frac{x_d}{|x_s|}-\frac{(x_s^*x_d+x_d^*x_s)x_s}{2|x_s|^3}\right)\epsilon,	& \text{ if }x_s\neq 0, \\
				\frac{x_d}{|x_d|}\epsilon,	& \text{ if }x_s=0,  x_d\neq 0,\\
				1,	& \text{ if }x_s=x_d=0.
			\end{array}
			\right. 
		\end{align} 
		\red{We stop the iterative process when $\|\vdq z(k+1)-\vdq z(k)\|_{2^R}$ is less than a predefined tolerance. Here, the $2^R$ norm is defined by \eqref{equ:2R-norm}.
			We summarize the numerical algorithm in Algorithm~\ref{Alg:controllaw}.  
			
			\begin{algorithm}[h]
				\caption{The projection iteration method}\label{Alg:controllaw}
				\begin{algorithmic}[1]
					\Require  a directed and connected graph $G = (V, E)$,   a
					desired relative configuration scheme $\left\{ \hat q_{d_{ij}} : (i, j) \in E \right\}$, stepsize $\alpha_k$, the precondition matrix $\dq K$, the tolerance $\delta$, and the maximum iteration $k_{\max}$. 
					\State Construct the Laplacian matrix by \eqref{Laplaican_mat}. 
					\For{$k=0,1,\dots,k_{\max}$}
					\State Compute $\vdq z(k+1)$ by \eqref{equ:proj_iterate}.
					\If{$\|\vdq z(k+1)-\vdq z(k)\|_{2^R}\le \delta$}
					\State Break.
					\EndIf
					\EndFor
					
					\State \textbf{Output:} $\vdq z(k+1)$. 
				\end{algorithmic}
			\end{algorithm}

			\section{Numerical Experiments}\label{sec:numer}
			In this section,  we perform   numerical simulations across three directed graph topologies: cycles, stars, and grids, to validate the effectiveness of the proposed control law presented in Algorithm~\ref{Alg:controllaw}. 
			For a directed cycle with $n$ vertices, the edge set is defined as: $E=\{(1,2),(2,3),\dots,(n-1,n),(n,1)\}$.  
			A star topology with $n=2n_0$ vertices has edges:  $E=\{(1,2),(2,3),\dots,(n_0-1,n_0),(n_0,1)\}\cup \{(i,n_0+i), (i,n_0+i+1)\}$ for all $i=1,\dots,n_0$ with $2n_0+1=1$. This structure combines a central cycle with outward branches. 
			For a grid with  $n = n_0^2$ vertices, the edge set is  $E=\{(jn_0+i,jn_0+i+1)\}\cup  \{((i-1)n0+j,in_0+j)\}$ for all $i=1,\dots,n_0-1, j =1,\dots,n_0$,  capturing both horizontal and vertical directed connections in a lattice. 
			For comparative analysis, we also generate their undirected counterparts by adding reverse edges with conjugate configuration weights.

			Given a formation of $n$ agents with arbitrary  initial configurations, the control objective is to achieve asymptotic convergence to the desired relative pose, specified by both rotational and translational states.  The desired configuration   $\vdq q_d\in{\hat{\mathbb U}^n}$ is generated as follows. 
			For any $i=1,\dots,n$, we begin by defining the angle as $\theta_i = \frac{i-1}{n}2\pi$ and the positions $\qu t_i$   as the locations corresponding to the graph topology, respectively. Let $\qu t_i = [0,\mathbf t_i]$ represent  the vector quaternion and let the axis be defined by $\mathbf v = [\cos(\theta_i),\sin(\theta_i),0]$.  The attitude  quaternion is then given by  $\qu q_{s_i} = \left[\cos\left(\frac{\theta_i}{2}\right),\sin\left(\frac{\theta_i}{2}\right){\mathbf v}\right]$ and the  desired  configuration of the $i$-th quadrotor is expressed as $\vdq q_{d_i} = \qu q_{s_i}+\frac12 \qu t_i\qu q_{s_i}\red{\epsilon},$ 
			respectively. 	 We consider two measurement scenarios in our analysis: noise-free measurements and noisy measurements. In the ideal noise-free case, the relative dual quaternion configuration between agents $i$ and $j$ is given by 
			\[\hat q_{d_{ij}} = \hat q_{d_i}^*\hat q_{d_j}, \quad \forall (i,j)\in E.\]
			For the more practical noisy measurement case, we model the relative configuration as
			\[\hat q_{d_{ij}} = \hat q_{d_i}^*\hat q_{d_j} \hat p_{ij}, \quad \forall (i,j)\in E,\]
			where $\hat p_{ij}$ is a multiplicative noise term represented by a unit dual quaternion close to the identity. This formulation accounts for measurement uncertainties while maintaining the unit dual quaternion constraint essential for rigid body configurations.

			We plot the desired configuration of  directed  cycles, stars, and grids in Fig.~\ref{fig:5p_cycle}.  
			
			\begin{figure}[h]
				\centering
				\begin{tabular}{ccc}
					\includegraphics[width=0.3\linewidth]{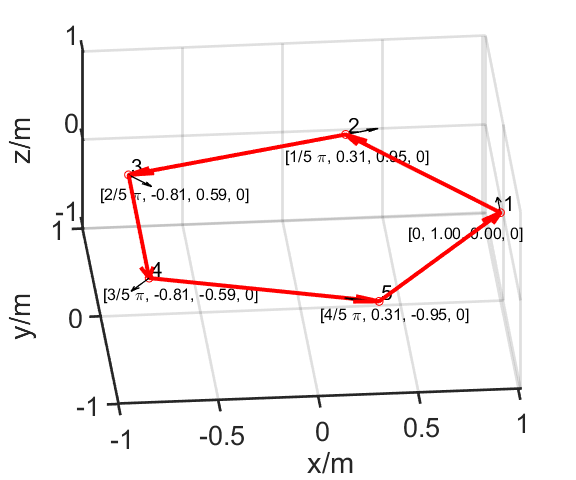} &
					\includegraphics[width=0.3\linewidth]{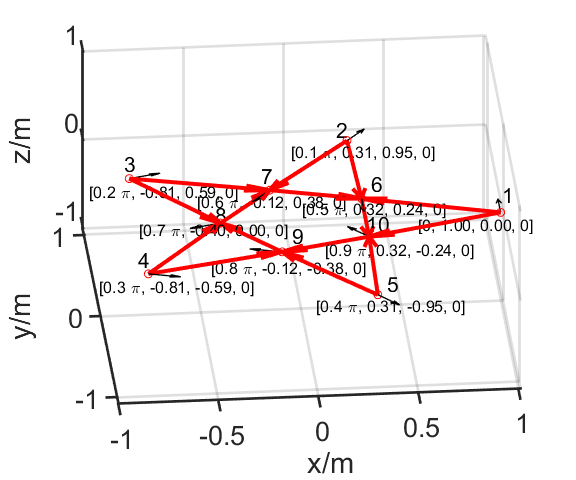} &
					\includegraphics[width=0.3\linewidth]{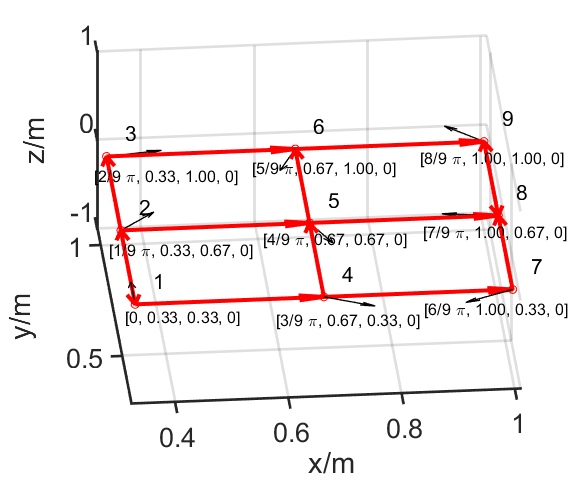} 
					\\
					(a)  Directed cycle ($n=5$) & (b)  Directed star  ($n=10$)  & (c) Directed grid  ($n=9$) 
				\end{tabular} 
				\caption{The directed  cycles, stars, and grids with  corresponding desired formation.}\label{fig:5p_cycle}
			\end{figure}
			
			\subsection{Noise-free measurements} 
			In this subsection, we examine the desired relative configurations that satisfy the geometric constraints specified in equation \eqref{desrel} under noise-free measurement conditions.}  
		By \eqref{L=QLQ}, the Laplacian matrix of the directed cycle corresponding to  Fig.~\ref{fig:5p_cycle}~(a) is presented as follows, 
		\[
		\dq L = \begin{bmatrix}
			1 & -\hat q_{d_1}^*\hat q_{d_2} & 0 & 0 & 0\\ 
			0 & 1 &  -\hat q_{d_2}^*\hat q_{d_3} & 0 & 0 \\
			0 & 0 & 1 &  -\hat q_{d_3}^*\hat q_{d_4} & 0 \\
			0 & 0 & 0 & 1 &  -\hat q_{d_4}^*\hat q_{d_5}\\
			-\hat q_{d_5}^*\hat q_{d_1}  	& 0 & 0 &0  & 1 \\
		\end{bmatrix}.
		\]
		Similarly, we may construct the Laplacian matrices of directed stars and grids.

		In Algorithm~\ref{Alg:controllaw}, we employ the parameter values  $\alpha_k=0.2$ and $ \dqm K=I$ for   
		all iterations. 
		\red{We set the residue tolerence as $\delta=10^{-15}\sqrt{n}$ unless explicitly noted otherwise.} 
		Denoting the final configuration after convergence as $\vdq z(k^1)$, we compute the transformation value $\dq c = ({\dq q}^*_{d1})^{-1}\dq z_1(k^1)$. Then  the desired asymptotic configuration is defined by 
		$\vdq z^{\infty}=\bar{\vdq q}_d\dq c$. 
		To quantify convergence, we evaluate the configuration error at each time slot $t$ using 
		\begin{equation}\label{err: config}
			\err(t) = \|\vdq z^{\infty}-\vdq z(t)\|_{2^R},
		\end{equation}
		where $\|\vdq q\|_{2^R}=\sqrt{\|{\bf q}_s\|^2+\|{\bf q}_d\|^2}.$

		\begin{table*}
			\centering
			
			\caption{Error evolution defined in \eqref{err: config} for undirected cycles (``cycle'') and directed  cycles (``d-cycle''). We consider errors below $10^{-16}$ to be negligible, represented as zero.} \label{tab:controllaw_cycle}
			
			\begin{tabular}{c|c| ccc|ccc  }
				\hline
				
				& &  \multicolumn{3}{c|}{cycle} &  \multicolumn{3}{c}{d-cycle}\\ \hline
				$n$	& &  $t = 30$ &   $t = 50$ &  $t = 70$ &  $t = 30$ &   $t = 50$ &  $t = 70$  \\ \hline
				\multirow{2}{*}{5}& $\err(t)$ &  5.85e$-$16 & 4.80e$-$16 & 4.80e$-$16 & 4.05e$-$07 & 1.59e$-$12 & 1.44e$-$15 \\ 
				&	$\exp(-\lambda_{2r}t)$ &	0 & 0 & 0 & 9.94e$-$10 & 9.90e$-$16 & 0 \\  \hline
				\multirow{2}{*}{7}& $\err(t)$ & 	2.35e$-$10 & 1.26e$-$15 & 1.26e$-$15 & 8.72e$-$04 & 1.15e$-$06 & 2.24e$-$09 \\ 
				&	$\exp(-\lambda_{2r}t)$ &	1.55e$-$10 & 0 & 0 & 1.24e$-$05 & 6.67e$-$09 & 3.58e$-$12 \\ \hline
				\multirow{2}{*}{9}& $\err(t)$ & 1.27e$-$05 & 6.86e$-$10 & 3.70e$-$14 & 7.15e$-$02 & 1.42e$-$03 & 2.87e$-$05 \\ 
				&	$\exp(-\lambda_{2r}t)$ &	8.01e$-$07 & 6.91e$-$11 & 5.96e$-$15 & 8.95e$-$04 & 8.31e$-$06 & 7.72e$-$08 \\  \hline
			\end{tabular}
		\end{table*}

		We consider three convergence scenarios by setting  $k_{\max}=150,250,350,$ corresponding to termination time    $t=30,50,70$, respectively, with a fixed  step size    $\alpha_k=0.2$. 
		\red{For comprehensive performance evaluation, we operate Algorithm~\ref{Alg:controllaw} without its default termination criterion, allowing observation of full convergence dynamics.}

		The results for agent populations $n=5,7,9$ under both directed and undirected cycle topologies are presented in Table~\ref{tab:controllaw_cycle}. 
		The results affirm   that the iterative sequence generated by our algorithm successfully converges to the desired configuration. Notably, the evolution of $\err(t)$ closely follows the exponential decay pattern described by $\exp(-\lambda_{2r}t)$, except when $\exp(-\lambda_{2r}t)$ falls below $10^{-16}$ while $\err(t)$ remains above this threshold,  primarily attributable to numerical error propagation. This verifies our   convergence rate analysis in Theorem~\ref{Thm:controllaw}. Additionally, the  real values of the  second smallest eigenvalues $\lambda_{2r}$ associated with Laplacian matrices of directed cycles are significantly smaller compared to their undirected counterparts, resulting in a reduced convergence rate. 
		The above conclusion is further confirmed by the stars and  grids, as demonstrated  in  Tables~\ref{tab:controllaw_star} and \ref{tab:controllaw_grid}.  
		
		\begin{table*}
			\centering
			
			\caption{Error evolution defined in \eqref{err: config} for undirected stars (``star'') and directed  stars (``d-star''). We consider errors below $10^{-16}$ to be negligible, represented as zero.} \label{tab:controllaw_star}
			
			\begin{tabular}{c|c| ccc|ccc  }
				\hline
				
				& &  \multicolumn{3}{c|}{star} &  \multicolumn{3}{c}{d-star}\\ \hline
				$n$	& &  $t = 30$ &   $t = 50$ &  $t = 70$ &  $t = 30$ &   $t = 50$ &  $t = 70$  \\ \hline
				\multirow{2}{*}{10}& $\err(t)$ &5.20e$-$13 & 7.31e$-$16 & 8.04e$-$16 & 5.22e$-$08 & 1.81e$-$13 & 1.03e$-$15 \\		 
				&	$\exp(-\lambda_{2r}t)$ & 7.29e$-$13 & 0 & 0& 9.94e$-$10 & 9.90e$-$16 & 0 
				\\ \hline
				
				\multirow{2}{*}{16}& $\err(t)$ & 4.43e$-$05 & 6.55e$-$09 & 9.69e$-$13 & 2.13e$-$02 & 1.65e$-$04 & 1.12e$-$06 \\ 
				&	$\exp(-\lambda_{2r}t)$ &3.17e$-$06 & 6.84e$-$10 & 1.48e$-$13 & 1.53e$-$04 & 4.36e$-$07 & 1.25e$-$09 \\ \hline
				\multirow{2}{*}{20}& $\err(t)$ &
				3.38e$-$03 & 1.08e$-$05 & 3.43e$-$08 & 1.89e$-$01 & 7.14e$-$03 & 2.85e$-$04 \\ 
				&	$\exp(-\lambda_{2r}t)$ & 2.30e$-$04 & 8.61e$-$07 & 3.23e$-$09 & 3.25e$-$03 & 7.13e$-$05 & 1.56e$-$06 \\  \hline
			\end{tabular}
		\end{table*}

		\begin{table*}
			\centering
			
			\caption{Error evolution defined in \eqref{err: config} for undirected grids (``grid'') and directed  grids (``d-grid'').  We consider errors below $10^{-16}$ to be negligible, represented as zero.} \label{tab:controllaw_grid}
			
			\begin{tabular}{c|c| ccc|ccc  }
				\hline
				
				& &  \multicolumn{3}{c|}{grid} &  \multicolumn{3}{c}{d-grid}\\ \hline
				$n$	& &  $t = 10$ &   $t = 30$ &  $t = 50$ &  $t = 10$ &   $t = 30$ &  $t = 50$  \\ \hline
				\multirow{2}{*}{16}& $\err(t)$ & 7.12e+00 & 1.21e$-$01 & 2.12e$-$03 & 3.95e+00 & 1.66e$-$07 & 1.12e$-$14\\ 
				&	$\exp(-\lambda_{2r}t)$ & 2.86e$-$03 & 2.33e$-$08 & 1.90e$-$13 & 4.54e$-$05 & 9.36e$-$14 & 0 \\ \hline 
				
				\multirow{2}{*}{49}& $\err(t)$ & 2.46e+00 & 4.42e$-$02 & 7.76e$-$04 & 6.24e+00 & 5.48e$-$07 & 9.57e$-$15 \\		 
				&	$\exp(-\lambda_{2r}t)$ & 1.38e$-$01 & 2.63e$-$03 & 5.00e$-$05 & 4.54e$-$05 & 9.36e$-$14 & 0  
				\\ \hline
				
				\multirow{2}{*}{100}& $\err(t)$ &
				1.42e+01 & 1.93e+00 & 2.65e$-$01 & 7.43e+00 & 5.34e$-$06 & 6.94e$-$14 \\ 
				&	$\exp(-\lambda_{2r}t)$ & 3.76e$-$01 & 5.30e$-$02 & 7.49e$-$03 & 4.54e$-$05 & 9.36e$-$14 & 0  \\  \hline
			\end{tabular}
		\end{table*}

		Furthermore, we present the iterative process   in   Fig.~\ref{fig:controllaw_cycle}, which    validates Algorithm~\ref{Alg:controllaw}'s $R$-linear convergence rate. \red{The iterative process of $\exp(-\lambda_{2r}t)$ is plotted as a theoretical reference ('Theory'). The exact agreement between this theoretical prediction and Algorithm~\ref{Alg:controllaw}'s convergence trajectory provides strong numerical evidence for the R-linear convergence rate proved in Theorem~\ref{Thm:controllaw}.}

		\begin{figure}[h!]
			\centering
			\begin{tabular}{ccc}
				\includegraphics[width=0.33\linewidth]{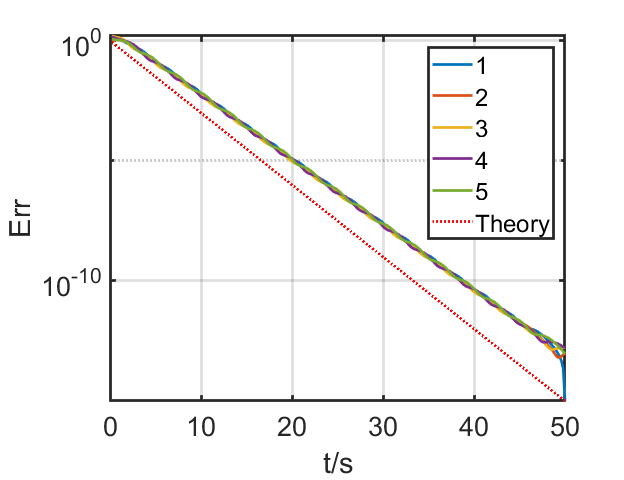} &
				\includegraphics[width=0.33\linewidth]{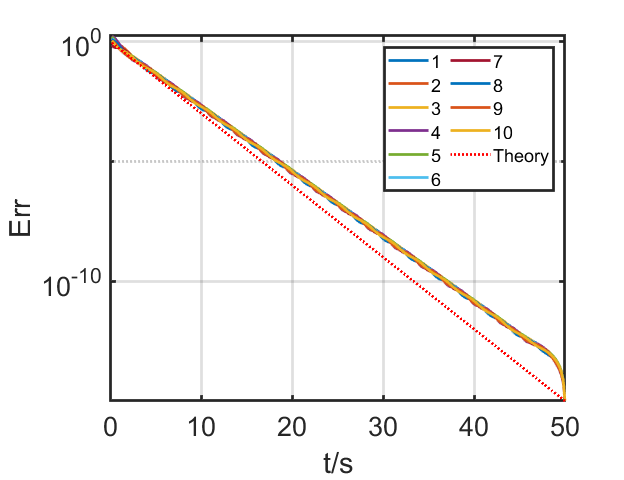} &	\includegraphics[width=0.33\linewidth]{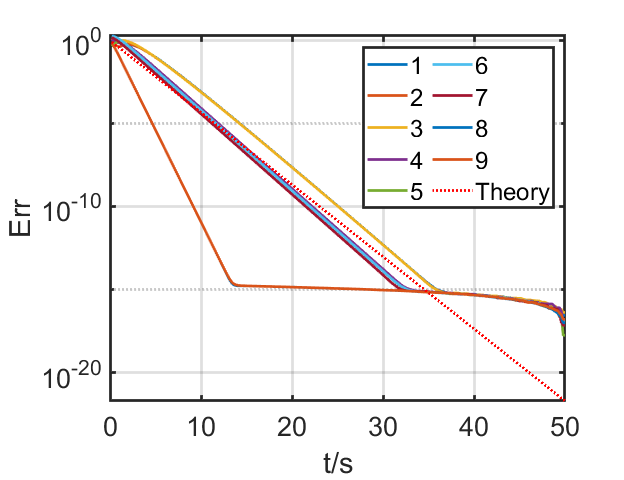} 
				\\
				(a)  Directed cycle ($n=5$) & (b)  Directed star  ($n=10$)  & (c) Directed grid  ($n=9$) 
			\end{tabular}
			\caption{Error evolution defined in \eqref{err: config} for  directed   cycle, star, and grid, respectively.}\label{fig:controllaw_cycle}
		\end{figure}

		At last,  Fig.~\ref{fig:5p_cycle_traj} illustrates the trajectories of the five quadrotors for the directed cycle, where their initial and terminal positions are respectively marked by blue circles and star symbols. Visual inspection clearly demonstrates that all quadrotors successfully converge to their predefined target configurations, as specified in  Fig.~\ref{fig:5p_cycle}(a). 
		
		\begin{figure}[!h]
			\begin{center}
				\includegraphics[width=0.8\linewidth]{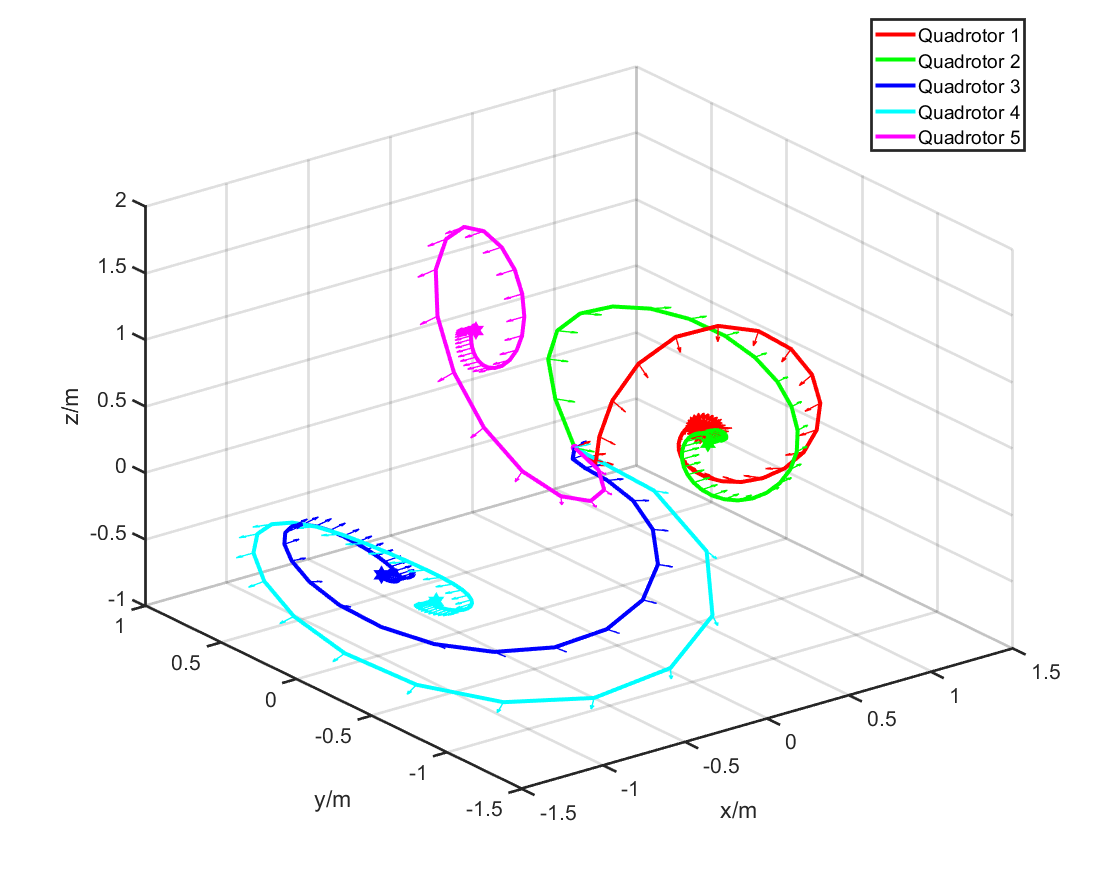}
			\end{center}
			\caption{The 3-D trajectories of five quadrotors. }\label{fig:5p_cycle_traj}
		\end{figure}

		\red{
			\subsection{With noise measurements}
			While Theorem~\ref{Thm:controllaw} assumes the desired relative configurations $\left\{ \hat q_{d_{ij}}: (i, j) \in E \right\}$ to be reasonable, this assumption may not always hold in practical scenarios. 
			The presence of sensor noise gives rise to  non-negligible disturbance terms in the system model. 
			The gain graph approach \cite{QC24} provides   a priori validation method for the configuration scheme's feasibility. 
			
			Suppose the prescribed relative configuration scheme  $\hat q_{d_{ij}}$ is infeasible, i.e.,  no set of absolute configurations $\hat q_{d_{i}}$ exists that satisfies the  conditions $\hat q_{d_{ij}} = \hat q_{d_i}^*\hat q_{d_j}$ for all $(i,j)\in E$.  In this case, we seek to compute the nearest feasible  configuration that minimizes the inconsistency in these relational constraints. 
			By \cite{QC24},  when the relative configuration is reasonable, it holds that  $\dq L\vdq x=0$ when $\vdq x=\bar{\vdq q}_d$. Consequently, we construct the following minimization  problem  
			\begin{equation}\label{equ:Lx=0}
				\min_{\vdq x} \| \dq L\vdq x\|_{2^R}^2 \text{\quad s.t. \quad} \vdq x\in\udqset^n.
			\end{equation}

			We follow  a two-stage computational process to compute \eqref{equ:Lx=0}.  First, we construct the dual quaternion Laplacian matrix and compute $\min_{\vdq x} \| \dq L\vdq x\|^2$ to obtain the intermediate solution. 
			This is followed by a projection stage to map each component    onto the unit dual quaternion manifold. Namely, 
			\begin{equation}
				\dq q_i={\rm Proj}_{\udqset} \left(\hat x_i\right),\ \  \forall \ i=1,\dots,n. 
			\end{equation}
			Here, $\mathrm{proj}_{\hat{\mathbb U}^n}(\cdot)$ denotes the elementwise projection to the set of unit dual quaternions  defined by \eqref{def:proj}. 
			
			Unlike in real vector spaces, the solution space of \eqref{equ:Lx=0} over dual quaternion rings exhibits non-uniqueness even when the nullspace is one-dimensional and $\vdq x\in\udqset^n$. Specifically, for any solution $\vdq x \in \udqset^{n}$ satisfying \eqref{equ:Lx=0} and any nonzero dual quaternion $\dq p\in\udqset$, the  vector $\vdq x \dq p\in\udqset^n$ remains a valid solution.  
			Therefore, we fix $\dq x_1=\dq 1$ and solve the reduced system.
			
			Denote $\dq L_{1}\in\dqset^{n}$ and $\dq L_{2}\in\dqset^{n\times (n-1)}$ as the first and  remaining columns of $\dq L$, respectively, and $\vdq x_{2}\in\dqset^{n-1}$ as the last $n-1$ elements of $\vdq x$.  
			We begin by fixing $x_{1s}=1$ and subsequently solving the resulting reduced system for the standard component as follows: 
			\begin{equation} \label{L_s_reduced}
				\min_{\vx_{2s}\in \mathbb{Q}^{n-1}}\quad  \|L_{2s} \vx_{2s} +L_{1s}\|^2. 
			\end{equation}
			Then we fix $x_{1d}=0$ and solve the resulting reduced system for the dual part as follows:  
			\begin{equation} \label{L_d_reduced}
				\min_{\vx_{2d}\in \mathbb{Q}^{n-1}}\quad 	\|L_{2s} \vx_{2d} + L_d\vx_s\|^2. 
			\end{equation}
			We denote $\vdq x=\vx_s+\vx_d\epsilon$    the nearest feasible  configuration that minimizes the inconsistency in these relational constraints.

			Finally, we present the results for the control law under noisy measurements and nearest desired relative configurations. 
			We evaluate the control law's performance with noisy measurements by generating perturbed relative configurations $\hat q_{d_{ij}} = \hat q_{d_i}^*\hat q_{d_j} \hat p_{ij}$ for all edges $(i,j)\in E$  as follows.  The measurement noise is modeled through two distinct components: Each angular parameter $\theta^k_{ij}$ is independently drawn from a uniform distribution over the interval  $[0,\sigma\pi]$, and the displacement vector $t_{ij}$ follows a zero-mean Gaussian distribution with variance $\sigma^2$. Then we obtain the standard parts  
			\[p_{ijs} = \cos(\theta^1_{ij}) + \sin(\theta^1_{ij})\cos(\theta^2_{ij})\ii+\sin(\theta^1_{ij})\sin(\theta^2_{ij})\cos(\theta^3_{ij})\jj+\sin(\theta^1_{ij})\sin(\theta^2_{ij})\sin(\theta^3_{ij})\kk,\] 
			and  the dual parts $p_{ijs}=\frac12 t_{ij}p_{ijs}$, respectively.

			We consider a directed grid network with 9 vertices under measurement noise with standard  deviation $\sigma=0.02$. Following the configuration approximation method described above,  we obtain estimated relative configurations with a residual norm of $0.0015$, demonstrating close agreement with the noise magnitude $\sigma$. The convergence behavior is illustrated in Fig.~\ref{fig:controllaw_grid_noise}, which reveals two distinct convergence regimes: The raw noisy measurements exhibit slow, sublinear convergence due to persistent measurement errors, while the nearest valid configurations achieve $R$-linear convergence, confirming the theoretical stability guarantees.  
			
			\begin{figure}[h!]
				\centering
				\begin{tabular}{cc} 
					\includegraphics[width=0.473\linewidth]{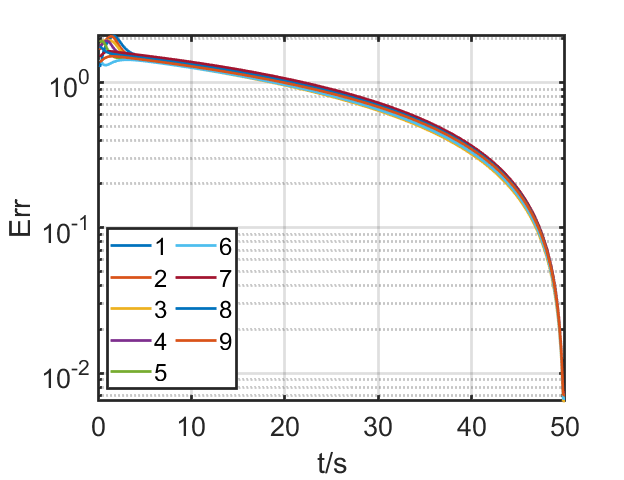} & 
					\includegraphics[width=0.473\linewidth]{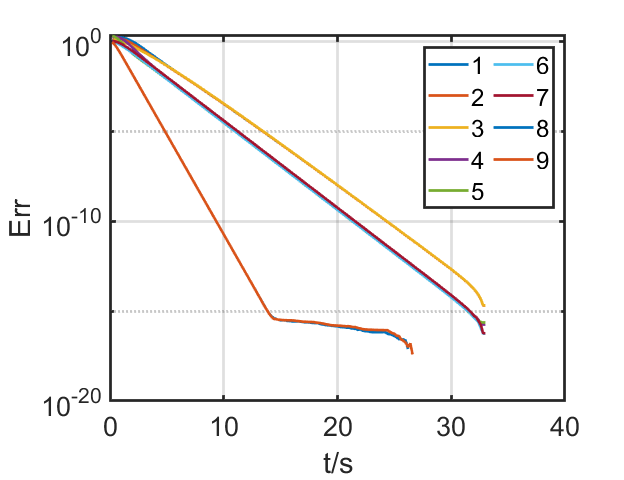} \\
					(a)  With noise measurements  & (b) Nearest desired relative configurations
				\end{tabular}
				\caption{Iterative process of error defined in \eqref{err: config} for  directed   grid  under noise measurements with $\sigma=0.02$ and nearest desired relative configurations,  respectively.}\label{fig:controllaw_grid_noise}
			\end{figure}
		}
		\section{Conclusions}\label{sec:conclu}
		In this paper, we introduced a control law that leverages the Laplacian matrix of   UDQDG and showed its R-linear convergence. This control mechanism can be seen as a directed generalization of the work presented in \cite{OM04}, adapted to the dual quaternion ring. Our proposed control law enables the simultaneous and coordinated control of both position and attitude for a group of 3D rigid bodies. It holds potential for addressing   formation control challenges   in    autonomous mobile robots, UAVs, AUVs,  and small satellites.
		
		Our primary focus was on directed networks with a fixed topology. However, the extension   to directed networks with switching topology and communication time-delays remains an area for future research.

		%
		

\end{document}